\documentclass[11pt,a4paper, final, twoside]{article}

\usepackage{amsmath}
\usepackage{fancyhdr}
\usepackage{amsthm}
\usepackage{amsfonts}
\usepackage{amssymb}
\usepackage{amscd}
\usepackage{latexsym}
\usepackage{amsthm}
\usepackage{graphicx}
\newcommand{\fact}[1]{#1\mathpunct{}!}
\usepackage{graphics}
\usepackage{afterpage}
\usepackage[colorlinks=true, urlcolor=blue,  linkcolor=black, citecolor=black]{hyperref}
\usepackage{color}
\theoremstyle{plain} 
\newtheorem{thm}{Theorem}[section]

\newtheorem{lem}[thm]{Lemma}

\theoremstyle{proposition}
\newtheorem{pr}{Proposition}[section]

\theoremstyle{definition}
\newtheorem{defi}{Definition}[section]
\theoremstyle{remark}

\numberwithin{equation}{section}
\begin{document}

\hyphenpenalty=100000
\begin{center}
{\Large {\textbf{\\ Existence and uniqueness of periodic solution of  nth-order Equations with delay in Banach space having Fourier type}}}\\[5mm]
{\large \textbf{{Bahloul Rachid}$^\mathrm{{\bf \color{red}{1}}}$\footnote{\emph{{1} : E-mail address : bahloul33r@hotmail.com}} } \\[1mm]
{\footnotesize $^\mathrm{}$ 
      }\\[3mm]}

\end{center}

\begin{flushleft}\footnotesize \it \textbf{$^\mathrm{{\bf \color{red}{1}}}$ Department of Mathematics, Faculty of Sciences and Technology, {\bf Fez}, Morocco.
}\\[3mm]
\end{flushleft}

\begin{center}\textbf {ABSTRACT}\end{center} 
{\footnotesize {{The aim of this work is to study the existence of a periodic solutions of   nth-order differential equations with delay $\frac{d}{dt}x(t) + \frac{d^{2}}{dt^{2}}x(t) + \frac{d^{3}}{dt^{3}}x(t) +...+ \frac{d^{n}}{dt^{n}}x(t) = Ax(t) + L(x_{t})+ f(t)$. Our approach is based on the  M-boundedness of linear operators, Fourier type, $B^{s}_{p, q}$-multipliers and Besov spaces.}}}

\afterpage{
\fancyhead{} \fancyfoot{}
\fancyfoot[R]{\footnotesize\thepage}
\fancyhead[R]{\scriptsize \it{ }
 }}

\section{Introduction}

Motivated by the fact that  neutral functional integro-differential equations  with finite delay arise in many areas of applied mathematics, this type of
equations has received much attention in recent years. In particular, the problem of existence of periodic solutions, has been considered by several authors. We refer the readers to papers [{\bf \color{green}{\cite{1}}},{\bf \color{green}{\cite{5}}}, {\bf \color{green}{\cite{7}}}, {\bf \color{green}{\cite{14}}}] and the references listed therein for information on this subject.\\
In this work, we study the existence of periodic solutions for the following integro-differential equations with delay

\begin{equation}\label{e1}
\displaystyle{\sum_{j=1}^{n}\frac{d^{j}}{dt^{j}}x(t) = \frac{d}{dt}x(t) + \frac{d^{2}}{dt^{2}}x(t) + \frac{d^{3}}{dt^{3}}x(t) +...+ \frac{d^{n}}{dt^{n}}x(t) = Ax(t) + L(x_{t})+ f(t)},
\end{equation}
where  $A : D(A) \subseteq  X \rightarrow X$ are a linear closed operators on  Banach space ($X, \left\|.\right\|$)  and
$f \in L^{p}(\mathbb{T}, X)$ for all $p \geq 1$. For $r_{2 \pi} := 2\pi N$  ( some $N \in \mathbb{N}$) $L$ is in $B(L^{p}([- r_{2 \pi} ,0],\ \ X);\ \ X)$  is the space of all bounded linear operators  and  $x_{t}$ is an element of $L^{p}([- r_{2 \pi}\ \ ,0],\ \ X)$  which is defined as follows $$x_{t}(\theta) = x(t+\theta)\;\; \text{for}\;\; \theta \in[- r_{2 \pi},\ \ 0].$$
In {\bf \color{green}{\cite{7}}}, Bahloul et al established the existence of a periodic solution for the following partial functional differential equation.
$$\frac{d}{dt}[x(t) - L(x_{t})]= A[x(t)- L(x_{t})]+G(x_{t})+f(t)$$
where  $A : D(A) \subseteq  X \rightarrow X$ is a linear closed operator on  Banach space ($X, \left\|.\right\|$) and $L$ and $G$ are in $B(L^{p}([- r_{2 \pi} ,0],\ \ X);\ \ X)$.\\
In {\bf \color{green}{\cite{1}}}, Arendt gave necessary and sufficient conditions for the existence of periodic solutions of  the following evolution equation.
$$\displaystyle{\frac{d}{dt}x(t)= Ax(t)+f(t)}\;\; \text{for}\;\; t \in \mathbb{R},$$
where $A$ is a closed linear operator on an UMD-space $Y$. \\
In {\bf \color{green}{\cite{13}}}, C. Lizama established results on the existence of periodic solutions of Eq. {\bf \color{red}{\eqref{e1}}} when  $L = 0,$  namely, for the following partial functional differential equation
$$\frac{d}{dt}x(t) = Ax(t)+G(x_{t})+f(t) \;\;\text{for}\;\; t \in \mathbb{R}$$
where $(A,D(A))$ is a  linear operator on an UMD-space $X$.\\
In {\bf \color{green}{\cite{12}}}, Hernan et al, studied the existence of periodic solution for the class of linear abstract neutral functional differential equation described in the following form:

\[\frac{d}{dt}[x(t) - Bx(t - r)]= Ax(t)+G(x_{t})+f(t)\;\;\text{ for}\;\; t \in \mathbb{R}\]
where $A: D(A) \rightarrow X$ and $B : D(B) \rightarrow X$ are closed linear operator such that $D(A) \subset D(B)$ and $G \in B(L^{p}([- 2\pi,\ \ 0],\ \ X);\ \ X)$.\\
This work is organized as follows : After preliminaries in the second section,  we give a main result and the conclusion.

\section{{\LARGE \bf vector-valued space and preliminaries}}
Let $X$ be a Banach Space. Firstly, we denote By $\mathbb{T}$ the group defined as the quotient $\mathbb{R}/2 \pi \mathbb{Z}$. There is an identification between functions on $\mathbb{T}$ and $2\pi$-periodic functions on $\mathbb{R}$. We consider the interval $[0, 2\pi$) as a model for $\mathbb{T}$ .\\
Given $1 \leq p < \infty $, we denote by $L^ {p}(\mathbb{T}; X)$ the space of $2\pi$-periodic locally $p$-integrable functions from $\mathbb{R}$ into $X$, with the norm:
\[\left\|f\right\|_{p}: =\left( \int_{0}^{2\pi} \left\|f(t)\right\|^{p}dt \right)^{1/p}\]
For  $f \in L^{p}(\mathbb{T}; X)$, we denote by $\hat{f}(k)$, $k \in \mathbb{Z}$ the $k$-th Fourier coefficient of $f$ that is defined by:

\[\mathcal{F}(f)(k) = \hat{f}(k) = \frac{1}{2\pi}\int_{0}^{2\pi}e^{-ikt}f(t)dt\;\; \text{for}\;\; k \in \mathbb{Z} \ \ \text{and} \  \ t \in \mathbb{R}.\]
For $1 \leq p < \infty$, the periodic vector-valued space is defined by.\\ 
Let $\mathcal{S}(\mathbb{R})$ be the Schwartz space of all rapidly decreasing smooth functions on $\mathbb{R}$.
Let D($\mathbb{T}$) be the space of all infinitely differentiable functions on $\mathbb{T}$ equipped with the locally convex topology given by the seminorms
$||f||_{n} = \sup_{x \in \mathbb{T}}|f^{(n)}(x)|$ for $n \in \mathbb{N}$. Let $D'(\mathbb{T}; X) = \mathcal{L}(D(\mathbb{T}), X)$. In order to
define Besov spaces, we consider the dyadic-like subsets of $\mathbb{R}$:
$$I_{0} = \{t \in \mathbb{R} : |t| \leq 2 \}, I_{k} = \{t \in \mathbb{R}, 2^{k-1} < |t| \leq 2^{k+1}\}$$
for $k \in \mathbb{N}$. Let $\phi(\mathbb{R})$ be the set of all systems $\phi = (\phi_{k})_{k \in \mathbb{N}} \subset \mathcal{S}(\mathbb{R})$ satisfying $supp(\phi_{k}) \subset \bar{I}_{k}$, 
for each $k \in \mathbb{N}, \sum_{k \in \mathbb{N}} \phi_{k}(x) = 1$.\\
Let $1 \leq p, q \leq \infty, s \in \mathbb{R}$ and $(\phi_{j})_{j \geq 0} \in \phi(\mathbb{R})$  the X-valued periodic Besov space is defined by
$$B_{p,q}^{s}(\mathbb{T}; X) = \{f \in D'(\mathbb{T}; X): ||f||_{B_{p,q}^{s}} := ( \sum_{j \geq 0}2^{sjq}||\sum_{k \in \mathbb{Z}}e_{k} \phi_{j}(k) \hat{f}(k)||^{q}_{p} )^{1/q} < \infty \}.$$

\begin{pr}\cite{14}\\
1) $B_{p,q}^{s}((0, 2\pi); X)$ is a Banach space;\\
2) Let $s > 0$. Then $f \in B_{p,q}^{s+1}((0, 2\pi); X)$ in and only if $f$ is differentiale  and $f' \in B_{p,q}^{s}((0, 2\pi); X)$ \\
3) Let $s > 0$. Then $f \in B_{p,q}^{s+j}((0, 2\pi); X)$ in and only if $f$ is differentiale four times and $f^{(j)} \in B_{p,q}^{s}((0, 2\pi); X)$ for all $j \in \mathbb{N}$
\end{pr}

\begin{defi}\cite{14}\\
For $1 \leq p < \infty$ , a sequence $\left\{M_{k}\right\}_{k \in Z} \subset B(X,Y)$ is a $B_{p,q}^{s}$-multiplier if for each $f \in B_{p,q}^{s}(\mathbb{T}, X),$ there exists  $u \in B_{p,q}^{s}(\mathbb{T},Y)$ such that $\hat{u}(k) = M_{k}\hat{f}(k)$ for all $k \in \mathbb{Z}$.
\end{defi}

\begin{defi}\cite{1}\\
The Banach space X has Fourier type $r \in ]1, 2]$ if there exists $C_{r} > 0$ such that
$$||\mathcal{F}(f)||_{r'} \leq C_{r}||f||_{r}, \  f \in L^{r}(\mathbb{R}, X)$$
where $\frac{1}{r'} + \frac{1}{r} = 1$.
\end{defi}

\begin{defi}\cite{14} \\
Let $\left\{M_{k}\right\}_{k \in Z} \subseteq B(X,Y)$  be a sequence of operators. $\left\{M_{k}\right\}_{k \in Z}$  is M-bounded of order 1( or M-bounded) if 
\begin{equation}\label{M}
\sup_{k}\|M_{k}\| < \infty \  \text{and}  \   \sup_{k}\|k(M_{k+1}-M_{k})\| < \infty
\end{equation}
\end{defi}

\begin{thm}\label{t41} \cite{1}\\
Let X and Y be Banach spaces having Fourier type $r \in ]1, 2]$ and let $\left\{M_{k}\right\}_{k \in Z} \subseteq B(X,Y)$ be a sequence satisfying (\ref{M}).
Then for $1 \leq p,q < \infty, s \in \mathbb{R}, \left\{M_{k}\right\}_{k \in Z}$ is an  $B_{p,q}^{s}$-multiplier.
\end{thm}

\begin{lem} \cite{7} \\ 
Let $L: L^{p}(\mathbb{T}, X) \rightarrow  X$  be a bounded linear operateur. Then
$$\widehat{L(u_{\textbf{.}})}(k) = L(e_{k}\hat{u}(k)):=L_{k}\hat{u}(k)\;\;  \text{for all}\;\; k \in \mathbb{Z}$$
and $\{ L_{k} \}_{k \in \mathbb{Z}}$ is r-bounded such that 
\begin{center}
$R_{p}((L_{k})_{k \in \mathbb{Z}})\leq (2r_{2\pi})^{1/p} \left\|L\right\|$.
\end{center}
\end{lem}

\section{Main result}
For convenience, we introduce the following notations: \\
$a_{k} = n+ \sum_{j=1}^{n-1}\sum_{p=1}^{j}C_{j}^{p}(ik)^{j+1-n-p}i^{p-1} + \sum_{p=2}^{n}C_{n}^{p}(ik)^{1-p}i^{p-1}, C_{n}^{p} = \frac{\fact{p}\fact{(n-p)}}{\fact{n}}$\\
$b_{k} = 2\sum_{p=1}^{n}C_{n}^{p}(ik)^{-p}i^{p} + \sum_{j=1}^{n-1}(ik)^{j-n}+ \sum_{j=1}^{n-1}\sum_{p=0}^{j}C_{j}^{p}(ik)^{j-p-n}i^{p}+ \sum_{p=1}^{n}C_{n}^{p}(ik)^{-p}i^{p}\sum_{j=1}^{n-1}(ik)^{j-n}$\\
$c_{k} = \sum_{p=0}^{n}C_{n}^{p}(ik)^{n-p}i^{p}\sum_{j=1}^{n}(ik)^{j}- (ik)^{n}\sum_{j=1}^{n}\sum_{p=0}^{j}C_{j}^{p}(ik)^{j-p}i^{p}$\\
$L_{k}(x):= L(e_{k}x)$ and $e_{k}(\theta):=e^{ik\theta}$, for all $k \in \mathbb{Z}$ and suppose that $\left\{L_{k}\right\}_{k \in Z}$  is M-bounded.
\begin{defi} :
Let $1 \leq p,q < \infty$ and $s > 0$. We say that a function $x \in  B_{p,q}^{s}(\mathbb{T}; X)$ is a strong  $B_{p,q}^{s}$-solution of (\ref{e1}) if $x(t)  \in D(A), x(t) \in B_{p,q}^{s+j}(\mathbb{T}; X), \ j \in \{1,...n\}$ and equation (\ref{e1}) holds for a.e $t \in \mathbb{T}$.
\end{defi}
We prove the following result.

\begin{lem}\label{l1}:
Let X be a Banach space and A be a  linear closed and bounded operator. Suppose that $(\sum_{j=1}^{n}(ik)^{j}I - A - L_{k})$ is bounded invertible and $(ik)^{n}(\sum_{j=1}^{n}(ik)^{j}I - A - L_{k})^{-1}$ is bounded. 
Then  \\ $\left\{N_{k} = (\sum_{j=1}^{n}(ik)^{j}I - A - L_{k})^{-1}\right\}_{k \in Z}, \left\{S_{k} = (ik)^{n}N_{k}\right\}_{k \in Z}$ and $\left\{T_{k} = L_{k}N_{k}\right\}_{k \in Z}$ are M-bounded.
\end{lem}
\begin{proof} The proof is given by several steps.\\
{\bf Step 1}: We claim that $\{kb_{k}\}_{k \in \mathbb{Z}}$ is bounded.\\
 We have
\begin{align*}
&kb_{k} = 2\sum_{p=1}^{n}C_{n}^{p}(ik)^{-p}i^{p} + \sum_{j=1}^{n-1}(ik)^{j-n}+ \sum_{j=1}^{n-1}\sum_{p=0}^{j}C_{j}^{p}(ik)^{j-p-n}i^{p} + \sum_{p=1}^{n}C_{n}^{p}(ik)^{-p}i^{p}\sum_{j=1}^{n-1}(ik)^{j-n}\\
&=2\sum_{p=1}^{n}C_{n}^{p}(ik)^{1-p}i^{p-1} - i\sum_{j=1}^{n-1}(ik)^{1+j-n} + \sum_{j=1}^{n-1}\sum_{p=0}^{j}C_{j}^{p}(ik)^{1+j-p-n}i^{p-1} - i\sum_{p=1}^{n}C_{n}^{p}(ik)^{-p}i^{p}\sum_{j=1}^{n-1}(ik)^{1+j-n}
\end{align*}
is bonded because:
\begin{align*}
&1 < p  \  \  \Rightarrow \  \ \sum_{p=1}^{n}C_{n}^{p}(ik)^{1-p}i^{p-1} \ \text{is bounded} \\
&1  + j< n \  \  \Rightarrow \  \  \sum_{j=1}^{n-1}(ik)^{1+j-n} \ \text{is bounded} \\
&1  + j< n \ \text{and} \ p \geq 0 \  \  \Rightarrow \  \  \sum_{j=1}^{n-1}\sum_{p=0}^{j}C_{j}^{p}(ik)^{1+j-p-n}i^{p-1} \ \text{is bounded} \\
&1  + j< n \ \text{and} \ p \geq 1 \  \  \Rightarrow \  \  \sum_{j=1}^{n-1}\sum_{p=0}^{j}C_{j}^{p}(ik)^{1+j-p-n}i^{p-1} \ \text{is bounded}
\end{align*}
{\bf Step 2}: New, we claim that, for all $k \in \mathbb{Z}$
\begin{equation}\label{c}
c_{k} = (ik)^{2n}b_{k} 
\end{equation}
\begin{align*}
&c_{k} = \sum_{p=0}^{n}C_{n}^{p}(ik)^{n-p}i^{p}\sum_{j=1}^{n}(ik)^{j}- (ik)^{n}\sum_{j=1}^{n}\sum_{p=0}^{j}C_{j}^{p}(ik)^{j-p}i^{p}\\
&= [ (ik)^{n} + \sum_{p=1}^{n}C_{n}^{p}(ik)^{n-p}i^{p} ][ \sum_{j=1}^{n-1}(ik)^{j} + (ik)^{n}] - (ik)^{n}[ \sum_{j=1}^{n-1}\sum_{p=0}^{j}C_{j}^{p}(ik)^{j-p}i^{p} + \sum_{p=0}^{n}C_{n}^{p}(ik)^{n-p}i^{p} ]\\
&= [ (ik)^{n} + \sum_{p=1}^{n}C_{n}^{p}(ik)^{n-p}i^{p} ][ \sum_{j=1}^{n-1}(ik)^{j} + (ik)^{n}] - (ik)^{n}[ \sum_{j=1}^{n-1}\sum_{p=0}^{j}C_{j}^{p}(ik)^{j-p}i^{p} + \sum_{p=1}^{n}C_{n}^{p}(ik)^{n-p}i^{p} + (ik)^{n}]\\
&=(ik)^{n}[\sum_{p=1}^{n}C_{n}^{p}(ik)^{n-p}i^{p} + \sum_{j=1}^{n-1}(ik)^{j}+ \sum_{j=1}^{n-1}\sum_{p=0}^{j}C_{j}^{p}(ik)^{j-p}i^{p} + \sum_{p=1}^{n}C_{n}^{p}(ik)^{n-p}i^{p}] + (\sum_{p=1}^{n}C_{n}^{p}(ik)^{n-p}i^{p})(\sum_{j=1}^{n-1}(ik)^{j})\\
&=(ik)^{2n}[2\sum_{p=1}^{n}C_{n}^{p}(ik)^{-p}i^{p} + \sum_{j=1}^{n-1}(ik)^{j-n}+ \sum_{j=1}^{n-1}\sum_{p=0}^{j}C_{j}^{p}(ik)^{j-p-n}i^{p}] + (ik)^{n}(\sum_{p=1}^{n}C_{n}^{p}(ik)^{-p}i^{p})(\sum_{j=1}^{n-1}(ik)^{j})\\
&=(ik)^{2n}[2\sum_{p=1}^{n}C_{n}^{p}(ik)^{-p}i^{p} + \sum_{j=1}^{n-1}(ik)^{j-n}+ \sum_{j=1}^{n-1}\sum_{p=0}^{j}C_{j}^{p}(ik)^{j-p-n}i^{p} ] + (ik)^{2n}(\sum_{p=1}^{n}C_{n}^{p}(ik)^{-p}i^{p})(\sum_{j=1}^{n-1}(ik)^{j-n})\\
&=(ik)^{2n}[2\sum_{p=1}^{n}C_{n}^{p}(ik)^{-p}i^{p} + \sum_{j=1}^{n-1}(ik)^{j-n}+ \sum_{j=1}^{n-1}\sum_{p=0}^{j}C_{j}^{p}(ik)^{j-p-n}i^{p}  + \sum_{p=1}^{n}C_{n}^{p}(ik)^{-p}i^{p}\sum_{j=1}^{n-1}(ik)^{j-n} ]\\
&=(ik)^{2n}b_{k}.
\end{align*}
{\bf Step 3}: We claim that 

\begin{eqnarray}\label{b1}
\left\{
\begin{array}{ccccc}
\sup_{k}\|k(N_{k+1} - N_{k})\| < \infty,\\\\
\sup_{k}\|k(S_{k+1} - S_{k})\| < \infty,\\\\
\sup_{k}\|k(T_{k+1} - T_{k})\| < \infty
\end{array}
\right.
\end{eqnarray}
By hypothesis we have, $\{ N_{k}\}_{k \in \mathbb{Z}}$ and $\{ S_{k}\}_{k \in \mathbb{Z}}$ are bounded. Then We have
\begin{align*}
&\sup_{k \in \mathbb{Z}}\| k (N_{k+1} - N_{k})\|=\sup_{k \in \mathbb{Z}}\left\|k N_{k+1}[(\sum_{j=1}^{n}(ik)^{j}I-A- L_{k}) -  ( \sum_{j=1}^{n}(i(k+1))^{j}I-A- L_{k+1})]N_{k}\right\|\\
&=\sup_{k \in \mathbb{Z}}\left\|k N_{k+1}[(\sum_{j=1}^{n}(ik)^{j}I-A- L_{k}) -  ( \sum_{j=1}^{n}\sum_{p=0}^{j}C_{j}^{p}(ik)^{j-p}i^{p}I-A- L_{k+1})]N_{k}\right\|\\
&=\sup_{k \in \mathbb{Z}}\left\| - kN_{k+1}[\sum_{j=1}^{n-1}\sum_{p=1}^{j}C_{j}^{p}(ik)^{j-p}i^{p} + \sum_{p=1}^{n}C_{n}^{p}(ik)^{n-p}i^{p}]N_{k} + N_{k+1}k(L_{k} -  L_{k+1})N_{k}\right\|\\ 
&=\sup_{k \in \mathbb{Z}}\left\| - kN_{k+1}[\sum_{j=1}^{n-1}\sum_{p=1}^{j}C_{j}^{p}(ik)^{j-p}i^{p} + \sum_{p=2}^{n}C_{n}^{p}(ik)^{n-p}i^{p}+in(ik)^{n-1}]N_{k} + N_{k+1}k(L_{k} -  L_{k+1})N_{k}\right\|\\ 
&=\sup_{k \in \mathbb{Z}}\left\| - kN_{k+1}[n+ \frac{1}{i(ik)^{n-1}}\sum_{j=1}^{n-1}\sum_{p=1}^{j}C_{j}^{p}(ik)^{j-p}i^{p} + \frac{1}{i(ik)^{n-1}}\sum_{p=2}^{n}C_{n}^{p}(ik)^{n-p}i^{p}]i(ik)^{n-1}N_{k} + N_{k+1}k(L_{k} -  L_{k+1})N_{k}\right\|\\ 
&=\sup_{k \in \mathbb{Z}}\left\| - N_{k+1}[n+ \sum_{j=1}^{n-1}\sum_{p=1}^{j}C_{j}^{p}(ik)^{j+1-n-p}i^{p-1} + \sum_{p=2}^{n}C_{n}^{p}(ik)^{1-p}i^{p-1}](ik)^{n}N_{k} + N_{k+1}k(L_{k} -  L_{k+1})N_{k}\right\|\\ 
&=\sup_{k \in \mathbb{Z}}\left\| - N_{k+1}[n+ \sum_{j=1}^{n-1}\sum_{p=1}^{j}C_{j}^{p}(ik)^{j+1-n-p}i^{p-1} + \sum_{p=2}^{n}C_{n}^{p}(ik)^{1-p}i^{p-1}]S_{k} + N_{k+1}k(L_{k} -  L_{k+1})N_{k}\right\|\\ 
&=\sup_{k \in \mathbb{Z}}\left\| -  N_{k+1} a_{k} S_{k}+ N_{k+1}k(L_{k} -  L_{k+1})N_{k}\right\|
\end{align*}
We obtain: 
\begin{equation}\label{b3}
 \sup_{k \in \mathbb{Z}}\| k (N_{k+1} - N_{k})\| < \infty 
\end{equation}
On the other hand,  we have 
\begin{align*}
&\sup_{k \in \mathbb{Z}}\| k (S_{k+1} - S_{k})\|=\sup_{k \in \mathbb{Z}}\left\|k[(i(k+1))^{n}N_{k+1} - (ik)^{n}N_{k}]\right\|\\
&=\sup_{k \in \mathbb{Z}}\left\|kN_{k+1}[(ik+i)^{n}(\sum_{j=1}^{n}(ik)^{j}I - A - L_{k}) - (ik)^{n}(\sum_{j=1}^{n}(i(k+1))^{j}I - A - L_{k+1})]N_{k}\right\|\\
&=\sup_{k \in \mathbb{Z}}\left\|kN_{k+1}[(ik+i)^{n}(\sum_{j=1}^{n}(ik)^{j}I - A - L_{k}) - (ik)^{n}(\sum_{j=1}^{n}(ik+i)^{j}I - A - L_{k+1})]N_{k}\right\|\\
&=\sup_{k \in \mathbb{Z}}\left\|kN_{k+1}[(ik+i)^{n}\sum_{j=1}^{n}(ik)^{j}I- (ik)^{n}\sum_{j=1}^{n}(ik+i)^{j}I- (ik+i)^{n}( A + L_{k}) + (ik)^{n} ( A + L_{k+1})]N_{k}\right\|\\
&=\sup_{k \in \mathbb{Z}}\left\|kN_{k+1}[\sum_{p=0}^{n}C_{n}^{p}(ik)^{n-p}i^{p}\sum_{j=1}^{n}(ik)^{j}I- (ik)^{n}\sum_{j=1}^{n}\sum_{p=0}^{j}C_{j}^{p}(ik)^{j-p}i^{p}I- \sum_{p=0}^{n}C_{n}^{p}(ik)^{n-p}i^{p}( A + L_{k}) + (ik)^{n} ( A + L_{k+1})]N_{k}\right\|\\
&=\sup_{k \in \mathbb{Z}}\left\|kN_{k+1}[c_{k}- \sum_{p=1}^{n}C_{n}^{p}(ik)^{n-p}i^{p}( A + L_{k}) + (ik)^{n} (L_{k+1} - L_{k})]N_{k}\right\|\\
&=\sup_{k \in \mathbb{Z}}\left\|kN_{k+1}[(ik)^{2n}b_{k}- \sum_{p=1}^{n}C_{n}^{p}(ik)^{n-p}i^{p}( A + L_{k}) + (ik)^{n} (L_{k+1} - L_{k})]N_{k}\right\| \ (\text{by} \ (\ref{c}) )\\
&=\sup_{k \in \mathbb{Z}}\left\|kN_{k+1}(ik)^{2n}b_{k}N_{k} - kN_{k+1}\sum_{p=1}^{n}C_{n}^{p}(ik)^{n-p}i^{p}( A + L_{k})N_{k} + N_{k+1}k(L_{k+1} - L_{k})S_{k}\right\|\\
&=\sup_{k \in \mathbb{Z}}\left\|kN_{k+1}(ik)^{n}b_{k}S_{k} - kN_{k+1}\sum_{p=1}^{n}C_{n}^{p}(ik)^{-p}i^{p}( A + L_{k})(ik)^{n}N_{k} + N_{k+1}k(L_{k+1} - L_{k})S_{k}\right\|\\
&=\sup_{k \in \mathbb{Z}}\left\|(\frac{ik}{ik+i})^{n}S_{k+1}kb_{k}S_{k} - \frac{k}{(ik+i)^{n}}S_{k+1}\sum_{p=1}^{n}C_{n}^{p}(ik)^{-p}i^{p}( A + L_{k})S_{k} + N_{k+1}k(L_{k+1} - L_{k})S_{k}\right\|\\
\end{align*}
Then  
\begin{center}$\sup_{k \in \mathbb{Z}}\| k (S_{k+1} - S_{k})\| < \infty$ \end{center}
Finally we have 
\begin{align*}
\sup_{k \in \mathbb{Z}}\left\|k (T_{k+1} - T_{k})\right\|&=\sup_{k \in \mathbb{Z}}\left\|(k[ L_{k+1}N_{k+1} - L_{k}N_{k}]\right\| \\
&=\sup_{k \in \mathbb{Z}}\left\|k( L_{k+1}- L_{k})N_{k+1} +  L_{k}k(N_{k+1}-N_{k}) \right\|\\
&\leq \sup_{k \in \mathbb{Z}}\left\|k( L_{k+1}- L_{k})N_{k+1}\right\| + \sup_{k \in \mathbb{Z}}\left\|L_{k}k(N_{k+1}-N_{k})\right\|\\
&\leq \sup_{k \in \mathbb{Z}}\left\|k( L_{k+1}- L_{k})N_{k+1}\right\| +  \sup_{k \in \mathbb{Z}}\left\|L_{k}\right\|\sup_{k \in \mathbb{Z}}\left\|k(N_{k+1}-N_{k})\right\|
\end{align*}
Then by (\ref{b3}) we have
\begin{center}$\sup_{k \in \mathbb{Z}}\| k (T_{k+1} - T_{k})\| < \infty$ \end{center}
so, $(N_{k})_{k \in \mathbb{Z}}, (S_{k})_{k \in \mathbb{Z}}$ and  $(T_{k})_{k \in \mathbb{Z}}$  are M-bounded. 
\end{proof}

\begin{thm}
Let $1 \leq p,q < \infty$ and $s > 0$. Let X be a Banach space having Fourier type $r \in ]1, 2]$ and A be a linear closed and bounded operator. If  Suppose that $(\sum_{j=1}^{n}(ik)^{j}I - A - L_{k})$ is bounded invertible and $(ik)^{n}(\sum_{j=1}^{n}(ik)^{j}I - A - L_{k})^{-1}$ is bounded. Then for every $f \in B^{s}_{p,q}(\mathbb{T}, X)$ there exist a unique strong $B^{s}_{p,q}$-solution of ({\bf \color{red}{\ref{e1}}}).
\end{thm}

\begin{proof} Define $P_{k} = \sum_{j=1}^{n}(ik)^{j}N_{k} = (1 + \sum_{j=1}^{n-1}(ik)^{j-n})S_{k}$ for $k \in \mathbb{Z}$. Since by  Lemma(\ref{l1}), $(N_{k})_{k \in \mathbb{Z}}, (S_{k})_{k \in \mathbb{Z}}, (T_{k})_{k \in \mathbb{Z}}$ and $(P_{k})_{k \in \mathbb{Z}}$ are M-bounded, we have  by Theorem \ref{t41} that  $(N_{k})_{k \in \mathbb{Z}}, (P_{k})_{k \in \mathbb{Z}}$ and $(T_{k})_{k \in \mathbb{Z}}$ are an $B_{p,q}^{s}$-multipliers. Since $P_{k} -  AN_{k} - L_{k}N_{k} = I$ (because   $( (\sum_{j=1}^{n}(ik)^{j}I -A- L_{k})N_{k} = I),$ we deduce $ AN_{k}$ is also an $B^{s}_{p,q}$-multiplicateur.\\
Now let $f \in B^{s}_{p,q}(\mathbb{T}, X)$. Then there exist $u, v, w, x \in  B^{s}_{p,q}(\mathbb{T}, X)$, such that\\
$\hat{u}(k) = N_{k}\hat{f}(k),  \hat{v}(k) = P_{k}\hat{f}(k), \hat{w}(k) = AN_{k}\hat{f}(k)$ and $\hat{x}(k) = T_{k}\hat{f}(k)$ for all $k \in \mathbb{Z}$. So,
We have $\hat{u}(k) \in D(A)$ and $A\hat{u}(k)= \hat{w}(k)$ for all $k \in \mathbb{Z}$, we deduce that  $u(t) \in D(A)$. On the other hand $\exists v \in  B^{s}_{p,q}(\mathbb{T}, X)$ such that $\hat{v}(k) = P_{k}\hat{f}(k)= \sum_{j=1}^{n}(ik)^{j}N_{k}\hat{f}(k)=\sum_{j=1}^{n}(ik)^{j} \hat{u}(k)$.  Then we obtain $\sum_{j=1}^{n}\frac{d^{j}}{dt^{j}}u(t) = v(t)$ a.e. Since  $u(t) \in B^{s+j}_{p,q}(\mathbb{T}, X), \ j \in \{1,...n\}$.\\
We have $\widehat{ \sum_{j=1}^{n}\frac{d^{j}}{dt^{j}}u}(k) = \sum_{j=1}^{n}(ik)^{j}\hat{u}(k)$ and $\widehat{Lu_{.}}(k)= L_{k}\hat{u}(k)$ for all $k \in \mathbb{Z}$, It follows from the identity 
$$\sum_{j=1}^{n}(ik)^{j}N_{k} - A N_{k} - L_{k}N_{k}= I$$ that
\begin{center}
$\sum_{j=1}^{n}\frac{d^{j}}{dt^{j}}u(t) = Au(t) + L(u_{t}) + f(t)$
\end{center}
For the uniqueness we suppose two solutions $u_{1}$ and $u_{2}$, then  $u = u_{1} - u_{2}$ is  strong $L^{p}$-solution of equation (\ref{e1}) corresponding to the function $f = 0$, taking Fourier transform, we get $(\sum_{j=1}^{n}(ik)^{j}I  - A - L_{k})\hat{u}(k) = 0$, which implies that $\hat{u}(k) = 0$ for all $k \in \mathbb{Z}$ and  $u(t) = 0$. Then $u_{1} = u_{2}$. The proof is completed.
\end{proof}


\section{Conclusion}
We are obtained necessary and sufficient conditions to guarantee existence and uniqueness of periodic solutions to the equation $\sum_{j=1}^{n}\frac{d^{j}}{dt^{j}}u(t) = Au(t) + L(x_{t}) + f(t)$ in terms of either the M-boundedness of the modified resolvent operator determined by the equation. Our results are obtained in the  Besov space.

\scriptsize\----------------------------------------------------------------------------------------------------------------------------------------\\\copyright
{\it{Copyright International Knowledge Press. All rights reserved.}}

\end{document}